\documentclass{IOS-Book-Article}     
\usepackage{amsmath, amsbsy,amsfonts,amsopn,amstext, euscript, amssymb,amsbsy,amsfonts,amsthm,latexsym,amsopn,amstext, amsxtra,euscript,amscd, color}
\usepackage{graphicx}

\input xy   \xyoption{all}
 
\usepackage{tikz}
\usetikzlibrary{shapes,snakes,calendar,matrix,backgrounds,folding}

\usepackage{amsmath, amsbsy,amsfonts,amsopn,amstext, euscript}
\usepackage{amscd, epic, amssymb}

\usepackage{amsxtra}

\usepackage[dvips]{rotating}

\usepackage{amsmath,amssymb,amsbsy,amsfonts,amsthm,latexsym,amsopn,amstext, amsxtra,euscript,amscd}

\usepackage{hyperref}

\hypersetup{
  colorlinks   = true, 
  urlcolor     = blue, 
  linkcolor    = blue, 
  citecolor   = red 
}

\theoremstyle{theorem}
\newtheorem{thm}{Theorem}    

\newtheorem{lemma}{Lemma}

\theoremstyle{definition}
\newtheorem{definition}{Definition}

\newtheorem{remark}{Remark}

\def\d{\delta}

\def\X{\mathcal X}

\newcommand\G{\bar G}

\newcommand\s{\sigma}

\def\<{\langle}
\def\>{\rangle}

\begin{document}
\begin{frontmatter}          
%
\title{Cyclic curves over the reals}

\runningtitle{Cyclic curves over the reals}


\author[A]{\fnms{M.} \snm{Izquierdo}} 
\author[B]{\fnms{T.} \snm{Shaska}}

\runningauthor{Izquierdo/Shaska}

\address[A]{Department of Mathematics, \\
Link\"oping University, Link\"oping, Sweden. \\
E-mail: milagros.izquierdo@liu.se}

\address[B]{Department of Mathematics and Statistics, \\
Oakland University, Rochester, MI, USA;\\  
E-mail: shaska@oakland.edu}

\begin{abstract}
In this survey paper we study the automorphism groups of real curves admitting a regular meromorphic function $f$ of degree $p$, so called real cyclic $p$-gonal curves.  When $p=2$ the automorphism groups of real hyperelliptic curves where given by Bujalance et al. in \cite{BCGG}.

\end{abstract}



\end{frontmatter}


\section{Introduction}

Let $p$ be a prime number. A closed Riemann surface $X$ that can be realised as a $p$-sheeted covering of the Riemann sphere (i.e the projective line) is called $p$-gonal, and such a covering is called a $p$-gonal morphism. If the $p$-gonal morphism is a regular cyclic covering the Riemann surface is called a cyclic $p$-gonal surface. For $p=2$ these surfaces are called hyperelliptic and the covering is induced by the hyperelliptic involution. A $p$-gonal Riemann surface is called {\it real p-gonal} if there is an anticonformal involution $\s$ of $X$ commuting with the $p$-gonal morphism. $p$-gonal and real $p$-gonal Riemann surfaces have been extensively studied (see \cite{Accola, ano05, BCGG, Kontogeorgis, GSS, BCG, BCI, Shaska, SaS, Wootton, Wootton2, Broughton, BW}), specially hyperelliptic surfaces because of their applications to codes and cryptography, see \cite{ES}. 

A Riemann surface $X$ is cyclic $p$-gonal if and only if it is represented by an algebraic curve given by an equation of the form  
\[ y^{p}=\prod (x-a_{i})...\prod (x-m_{j})^{p-1}.\]
The projection $ (x,y) \to x$ is the $p$-gonal morphism. Let $\omega = \exp{{2\, i\pi\over p}}$. The automorphism of $X$ defined by $(x,y) \to (x, \omega y)$ generates the deck-transformation group of $p$-gonal morphism.

Let $X$ be a Riemann surface of genus $g\ge 2$. A symmetry $\s$ of $X$ is an anticonformal involution of $X$. The topological type of the symmetry $\s$ is given by the number of connected components, {\it ovals}, in the fixed-point set  $Fix(\s)$ and the orientability of the Klein surface $X/ \< \s\>$. We say that a symmetry $\s$ of a Riemann surface $X$ has {\it species} $0, +k$ or $-k$ (we denote $sp(\s)=0, \, sp(\s)=+k$, respectively $sp(\s)=+k$) if $\s$ has no fixed points, the fixed-point set of $\s$ consists of $k$ separating ovals or the fixed-point set consists of $k$ non-separating ovals respectively. The quotient Klein surface $X/ \< \s\>$ can be represented by a real curve where each oval corresponds to a componet of the real curve. Each real curve is a quotient $X/ \< \s\>$ and two real curves are birationally equivalent (isomorphic) if and only if they are quotients of a real Riemann surface $X$ by conjugate symmetries in $Aut^{\pm}(X)$.  A Riemann surface equipped with a symmetry $\s$ is called a {\it symmetric} or {\it real}  Riemann surface. A Riemann surface is a real Riemann surface if and only if it is represented by an algebraic curve with real equation. The complex conjugation induces a symmetry on the real Riemann surface.

The aim of this article is to give the automorphism groups of cyclic $p$-gonal real curves for $p$ an odd prime number. The case of hyperelliptic real curves was done by Bujalance et al. \cite{BCGG}. First of all,  a real curve $\mathcal{C} $ is the quotient a real Riemann surface $X$ by a symmetry $\s$ in a determined conjugacy class of symmetries;  real curves are in bijection with pairs  $(X, \s)$. To calculate the automorphism groups of real curves we use the extended automorphism groups $Aut^{\pm}(X)$ of Riemann surfaces obtained by Bartolini-Costa-Izquierdo (\cite{BCI}). Now, we shall consider a representative $\s$ of the conjugacy class of symmetries of $X$ yielding $\mathcal{C}$. The automorphism group of the real curve is the normalizer of $\< \s \>$ in $Aut^{\pm}(X)$.

\section{Trigonal Real Riemann surfaces and Fuchsian and NEC groups}

Let $\X_g$ be a compact Riemann surface of genus $g\geq 2$. The surface  $\X_g$ can be represented as a quotient $\X_g=\mathbb{H}/\Gamma^{+} $ of the
hyperbolic plane $\mathbb{H}$ under the action of a Fuchsian group $\Gamma^{+}   $, that is, a cocompact orientation-preserving subgroup of the group 
$\mathcal{G}=\mathrm{Aut}(\mathbb{H})$ of conformal and anticonformal
automorphisms of $\mathbb{H}$. A discrete, cocompact subgroup $\Gamma $ of 
$\mathrm{Aut}(\mathcal{D})$ is called an \textit{NEC (non-euclidean
crystallographic) group}. The subgroup of $\Gamma $ consisting of the
orientation-preserving elements is called the \textit{canonical Fuchsian
subgroup of} $\Gamma $. The algebraic structure of an NEC group and the
geometric and topological structure of its quotient orbifold are given by
the signature of $\Gamma$: 
\begin{equation}
s({\Gamma })=(h;\pm ;[m_{1}, \dots , m_{r}];\{(n_{11}, \dots , n_{1s_{1}}), \dots , (n_{k1}, \dots , n_{ks_{k}})\}).
\label{sign}
\end{equation}
The orbifold $\mathbb{H}/\Gamma $ has underlying surface of genus $h$, having $r$ cone points and $k$ boundary components, each with  $s_{j}\geq 0$ corner points, $j=1, \dots , k$. The signs $^{\prime \prime }+^{\prime \prime }$ and $^{\prime \prime }-^{\prime \prime }$ correspond to orientable and non-orientable orbifolds respectively. The integers $m_{i}$
are called the proper periods of ${\Gamma }$ and they are the orders of the cone points of $\mathbb{H}/\Gamma $. The brackets $(n_{i1}, \dots , n_{is_{i}})$
are the period cycles of ${\Gamma .}$ The integers $n_{ij}$ are the link periods of ${\Gamma }$ and the orders of the corner points of $\mathbb{H}/\Gamma $. The group $\Gamma $ is called the \textit{(orbifold-) fundamental group} of
 $\mathbb{H}/\Gamma $.

A group $\Gamma $ with signature (\ref{sign}) has a \textit{canonical presentation} with generators:
\begin{center}
$x_{1}, \dots , x_{r},e_{1}, \dots , ,e_{k},c_{ij},1\leq i\leq k,1\leq j\leq s_{i}+1$

and $a_{1},b_{1}, \dots , a_{h},b_{h}$ if $\mathbb{H}/\Gamma $ is orientable or $ d_{1}, \dots , d_{h}$ otherwise,
\end{center}
and relators:
\begin{center}
$x_{i}^{m_{i}},i=1, \dots , r,c_{ij}^{2},(c_{ij-1}c_{ij})^{n_{ij}},c_{i0}e_{i}^{-1}c_{is_{i}}e_{i}  $, $i=1, \dots , k,j=2, \dots , s_{i}+1$
\end{center}
and 
$x_{1}...x_{r}e_{1}...e_{k}a_{1}b_{1}a_{1}^{-1}b_{1}^{-1}...a_{h}b_{h}a_{h}^{-1}b_{h}^{-1}$ or $x_{1}...x_{r}e_{1}...e_{k}d_{1}^{2}...d_{h}^{2}$, according to whether  $\mathbb{H}/\Gamma $ is orientable or not.

The hyperbolic area of the orbifold $\mathbb{H}/\Gamma $ coincides with the hyperbolic area of an arbitrary fundamental region of $\Gamma $ and equals: 
\begin{equation}
\mu (\Gamma )=2\pi \left(\varepsilon h-2+k+\sum_{i=1}^{r}{ \left(1-{\frac{1}{{m_{i}}}} \right)} +{\frac{1}{2}}\sum_{i=1}^{k}\sum_{j=1}^{s_{i}}{ \left(1-{\frac{1}{{n_{ij}}}}\right)} \right),
\label{area}
\end{equation}
where $\varepsilon =2$ if there is a $^{\prime \prime }+^{\prime \prime }$
sign and $\varepsilon =1$ otherwise. If $\Gamma ^{\prime }$ is a subgroup of 
$\Gamma $ of finite index then $\Gamma ^{\prime }$ is an NEC group and the
following Riemann-Hurwitz formula holds: 
\begin{equation}
\lbrack \Gamma :\Gamma ^{\prime }]=\mu (\Gamma ^{\prime })/\mu (\Gamma ).
\label{HR}
\end{equation}

An NEC or Fuchsian group $\Gamma $ without elliptic elements is called a \textit{surface} NEC or Fuchsian group, and it has signature $(h;\pm;[-],\{(-),\overset{k}{\ldots },(-)\})$ or $(h;[-])$. Given a Riemann (resp. Klein) surface represented as the orbifold $X=\mathbb{H}/\Gamma $, with  $\Gamma $ a  surface Fuchsian (resp. NEC) group, a finite group $G$ is a group
of automorphisms of $X$ if and only if there exists an NEC group $\Delta$ and an epimorphism $\theta :\Delta \to G$ with $\ker (\theta)=\Gamma $, see \cite{BEGG}. The NEC group $\Delta $ is the lifting of $G$ to the universal covering $\pi : \mathbb{H}\rightarrow \mathbb{H}/\Gamma $ and is called the  \textit{universal covering transformations group} of $(X,G)$.

\begin{definition}
For a prime $p$, a real cyclic $p$-gonal Riemann surface is a triple $(X,f,\s )$ where $\s $ is a symmetry of $X$, $f$ is a cyclic $p$-gonal morphism and $f\circ \s =c\circ f$, and $c$ is the complex
conjugation.
\end{definition}

Notice that by Lemma 2.1 in \cite{Accola} the condition $f\circ \s =c\circ f$ is automatically satisfied for genera $g\geq (p-1)^{2}+1$, since the $p$-gonal morphism is unique. From now on, the genera will satisfy $g\geq (p-1)^{2}+1$; as a consequence of that, the group $C_{p}$ of deck-transformations of the $p$-gonal morphism, {\it the group of $p$-gonality}, is a normal subgroup of $Aut^{\pm}(\X_g)$, see \cite{Wootton} and \cite{CI2}.  

If a (surface) Fuchsian group $\Gamma^{+}$ uniformizes a real Riemann surface (complex curve) $X$ admitting a symmetry $\s$, then the NEC group $\Gamma =\< \Gamma^{+}, \s \>$ is the uniformizing group of the Klein surface (real curve) $X/\s$. 
 

We give now a characterization of real cyclic $p$-gonal Riemann surfaces represented by means of NEC groups.  The following is a straightforward generalization of the characterization theorems for real cyclic trigonal Riemann surfaces given in \cite{CI}, see also \cite{CI2}. 

\begin{thm}\label{thm1}
1)  Let $p$ be a prime number  and let $X$ be a Riemann surface with genus $g\geq (p-1)^{2}+1$. A Riemann surface $X$ admits a meromorphic function $f$ such that $(X, f)$ is a cyclic $p$-gonal Riemann surface if and only if there is a Fuchsian group $\Delta ^{+}$ with signature 
\begin{equation}
(0,[p,p,\overset{u}{...},p]),  \label{cyclic}
\end{equation}
with $(p-1)(u-2)=2g$, and a normal surface subgroup $\Gamma $ of index $p$ in $\Delta ^{+}$, such that $\Gamma $ uniformizes $X$.

2)  Let $p$ be a prime number and let $X$ be a Riemann surface with genus $g\geq (p-1)^{2}+1$. The surface $X$ admits a symmetry $\s $ and a meromorphic function $f$ such that $(X,f,\s )$
is a real cyclic $p$-gonal Riemann surface if and only if there is an NEC group $\Delta $ with signature 
\begin{equation}
(0,+,[p,p,\overset{u}{...},p],\{(p,p,\overset{v}{...},p)\}),
\label{realcyclic}
\end{equation}
where $(p-1)(2u+v-2)=2g$, \textit{such that there is an epimorphism }$\theta :\Delta \rightarrow G$\textit{, with }$G$\textit{\ isomorphic to either } 
$D_{p}=C_p \rtimes C_2=\left\< r,s:r^{p}=s^{2}=rsrs=1\right\> $\textit{\ or } $C_{p}\times C_2=\left\< r,s:r^{p}=s^{2}=rsr^{-1}s=1\right\> $
\textit{, where }$X$\textit{\ is conformally equivalent to }$\mathbb{H}/\ker \theta $
\textit{\ and }$\ker \theta $\textit{\ is a surface Fuchsian group.}
\end{thm}

\begin{remark}
i)  Observe that in the for real hyperelliptic curves there is only one case since $D_2=C_{2}\times C_2$.

ii)  In the case of $C_{2p} =\< C_{p}, \s \>$, then the signature \ref{realcyclic} of the NEC group $\Delta$ becomes  $(0,+,[p,p,\overset{u}{...},p],\{(-)\})$, with $(p-1)(u-1)=g$; then the genus $g$ is even.
\end{remark}

\begin{remark}
Let $p$ be an odd prime number. Consider a real cyclic p-gonal Riemann surface $(X, f, \s)$ with $p$-gonality group $C_p=\< \varphi \>$ where $\varphi: (x,y) \to (x, \omega y)$, with $\omega=\exp{{2\, i\pi\over p}}$, and equation $y^p=p(x)$. 
In the case $\< \varphi, \s \> = D_{p}$, the symmetry $\s$ is defined by $\s: (x,y) \to (\overline{x}, \overline{y})$ and $\overline{y}^p =p(\overline{x})=\overline{p(x)}$, that is $p(x)$ is a {\bf real polynomial}.
However, in the case $\< \varphi, \s \> = C_{2p}$ the symmetry $\s$ becomes $\s: (x,y) \to (\overline{x}, y)$. thus we have $y^p =p(\overline{x})=p(x)$. Complex polynomials $p(x)$ satisfying $p(\overline{x})=p(x)$ are {\bf not real} in general.
\end{remark}

The study of symmetry type of Riemann surfaces goes back to Klein. Let $(X,f,\s )$ be a real cyclic $p$-gonal Riemann surface. The $p$-gonality of $X$ imposes severe constrains on the type of symmetries $X$ admits. For the case of real hyperelliptic Riemann surfaces the  symmetry type was completely studied by Bujalance et al \cite{BCGG} since they gave the list of automorphism groups of hyperelliptic real curves depending on the ramification. 
In the case of an odd prime $p$ we have:

\begin{thm}\label{thm2}
Let $(X,f,\s )$ be a real cyclic $p$-gonal surface, with $X$ of genus 
$g\geq (p-1)^2+1$. Then the possible species of $\s $ are: $-1$ or $-p$ for odd genus $g$
and $\pm p$  and $\pm 1$ for even genus $g$.
\end{thm}

\begin{proof}
By Theorem~\ref{thm1}, there exists an NEC\ group $\Delta $ with signature 
\begin{equation*}
(0,+,[p,3,\overset{u}{...},p],\{(p,p,\overset{v}{...},p)\}),
\end{equation*}
with $(p-1)(2u+v-2)=2g$, and an epimorphism $\theta :\Delta \rightarrow G$, where $G$\ is isomorphic to either  $D_{p}$\ or $C_{2p}$, and $X$ is conformally equivalent to $\mathbb{H}/\ker \theta $, with $\ker \theta $ a surface Fuchsian group. We consider two cases:

\noindent \textbf{Case 1}. $G=D_{p}=\< r,s:r^{p}=s^{2}=rsrs=1\>$. The species of $\s $ is determined by the signature of the group $\theta ^{-1}(\< s\> )=\Lambda $ since $X/\<
\s \> =\mathbb{H}/\Lambda .$ If there are periods, they are odd, thus by \cite{BCS}, there is a unique conjugacy class of generating reflections in $\Lambda $, so the species of $\s $ is $\pm 1$. If there are no link periods and the connecting generator $e$ is mapped to the identity then the symmetry has $p$ ovals (see \cite{IS}). The sign in the species of $\s $ is $-$ if
and only if there is a cycle (not a loop) of odd length in the Schreier coset graph $\mathcal{G}$  of  $\< s \> $ in $D_{p}$ given by the action $\theta $ of $\Delta$ on the cosets $\left\{ 0=[s]  1=[rs] ,\dots, p-1=[r^{p-1}s] \right\} $ whose edges are produced by generators of $\Delta$ that are orientation reversing transformations, see Theorem 2 of \cite{HS}.  Such a cycle does exist if and only if at least one of the following three conditions are satisfied:

1. The signature of $\Delta $ contains proper periods,

2. The epimorphism $\theta $ maps the connecting generator $e$ of $\Delta $
to an element different from the identity,

3. The generating reflections of $\Delta $ are mapped by $\theta $ on more
than two involutions of $D_{p}$. \\

\noindent If one of the above conditions is satisfied, the species of $\s$ is $-1$.
If none of the conditions are satisfied, then $u=0, \, v\equiv 0\mod 2 $ in signature \ref{realcyclic}; the genus $(p-1)(v-2)=2g$ of the surface $X$ is even. \\

\noindent \textbf{Case 2}. $G=C_{2p}=\left\<
r,s:r^{p}=s^{2}=rsr^{-1}s=1\right\> $. Again, the species of $\s $ is
determined by the signature of the group $\theta ^{-1}(\left\<
s\right\> )=\Lambda .$ In this case the NEC group $\Delta $ must have
signature $(0,+,[p,p,\overset{u}{...},p],\{(-)\})$, with $(p-1)(u-1)=g$, and then
the surface $X$ must have even genus. $\s $ has is $ 1$ oval if 
$\theta (e)\neq 1$, and $ p$ ovals if $\theta (e)=1.$ Since the index of $\Lambda $ in $\Delta $ is $p$ and so odd, the
symmetry $\s $ has sign $+$ in its species by Corollary 2 of \cite{HS}. 


Now we give examples of real cyclic $p$-gonal surfaces with desired species of the symmetry . Let $\Delta $ be an NEC group with
signature $(0,+,[p,p],\{(p,\overset{{2(g-p+1)\over p-1}}{...},p)\}).$ and consider the epimorphism 
\[ \theta _{1}:\Delta \rightarrow \left\< r,s:r^{p}=s^{2}=rsrs=1\right\>, \]
defined by:
\[ \theta _{1}(x_{1})=r, \, \theta _{1}(c_{1})=s, \, \theta _{1}(c_{2i})=sr, \, \theta _{1}(c_{2i+1})=s, \,  \theta _{1}(e)=r^{\varepsilon }, \, \theta _{1}(x_{2})=r^{\mu }\]
  such that $\theta_{1}(c_{1}e^{-1}c_{{(2g-p+1)\over p-1}}e)=1$
and    $1+\varepsilon +\mu \equiv 0\mod p $.

By Theorem~\ref{thm1}, $\ker (\theta_{1})$ uniformizes a real cyclic $p$-gonal surface with a symmetry of
species $-1$.

Now consider an NEC group $\Delta$ with signature $(0,+,[-],\{(p,\overset{{2(g+p-1)\over p-1}}{...},p)\})$, with $g\equiv 0\mod 2 $ and
epimorphism 
\[ \theta _{2}:\Delta \rightarrow \left\< r,s:r^{3}=s^{2}=rsrs=1\right\>, \]
defined by: 
\[ \theta _{2}(c_{2i-1})=s, \, \theta _{2}(c_{2i})=sr, \, 1\le i \le {(g+p-1)\over p-1},\theta_{2}(e)=1.\]
By Theorem~\ref{thm1}, $\ker (\theta_{2})$ uniformizes a real cyclic $p$-gonal surface with a symmetry of species $+1$.

Consider now NEC groups with signature 
\[ (0,+,[p,p,\overset{{(g+p-1)\over p-1}}{...} , p],\{(-)\}),\]
with $g\equiv 0\mod 2 $, and epimorphisms $\theta _{3}:\Delta \rightarrow D_{p} $
sending the generating  reflection $c$ of $\Delta$ to $s$ and the generating elliptic elements to suitable powers of $r$ such that the the generator $e$ is sent to the identity. By Theorem~\ref{thm1}, $\ker (\theta_{2})$ uniformizes a real cyclic $p$-gonal surface with a symmetry of species $-p$.

Finally  consider NEC groups with signature 
\[ (0,+,[p,p,\overset{{(g+p-1)\over p-1}}{...} ,p],\{(-)\}),\]
with $g\equiv 0\mod 2 $, and epimorphisms $\theta _{3}:\Delta \rightarrow C_{2p} $ sending the generating  reflection $c$ of $\Delta$ to $s$ and the generating elliptic elements to suitable powers of $r$ such that the the generator $e$ is sent to either the identity, or $e$ is sent to a power of $r$.
By Theorem~\ref{thm1}, $\ker (\theta_{3})$ uniformizes a real cyclic $p$-gonal surface with a symmetry with desired species
\end{proof}


\section{Automorphism Groups of Real Curves}

Since 2001 Bujalance et al, Kontogeorgis, Wootton, Shaska, Sanjeewa and others have studied the groups of automorphisms of cyclic p-gonal Riemann surfaces allowing to find defining equations of the curves, see \cite{BCGG, Kontogeorgis, GSS, Shaska, Wootton, Wootton2, SaS, ano05, BCG, BCI}.  Gutierrez-Sevilla and Shaska \cite{GSS} found rational models of hyperelliptic curves. Shaska and Thompson \cite{ST}  have studied sub loci of hyperelliptic curves of genus 3.  

In 2001 Bujalance et al \cite{BCGG} gave the automorphism groups of hyperelliptic real curves. The gave: 
\begin{enumerate}
\item list of all groups which act as the full group $Aut^{\pm}(X)$ of  real hyperelliptic Riemann surfaces $X$ of genus $g\ge 2$
\item the classification of the symmetry types of all compact hyperelliptic Riemann surfaces of genus $g\ge 2$
\item explicit polynomial equations and explicit formulae for the hyperelliptic real curves in each symmetry class. 
\end{enumerate}

In the sequel $p$ is an odd prime integer and the genus $g \ge (p-1)^{2}+1$.
Consider a real curve $\mathcal{C}$ of genus $g$ which is cyclic $p$-gonal, that is, the real curve $\mathcal{C}$ can be considered to be the quotient $X/ \< \s \>$ of a real cyclic $p$-gonal Riemann surface $X$ by a symmetry $\s$ of $X$ in a given conjugacy class of symmetries in $Aut^{\pm}(X)$. As we say above the automorphism group of the real curve $\mathcal{C}$ is given by the (conformal and anticonformal) automorphisms of $X$ commuting with the symmetry $\s$, that is the normalizer in $Aut^{\pm}(X)$ of $\< \s \>$.

 Given a real cyclic $p$-gonal surface $(X,f,\s)$, we shall call
$\pm$-automorphism group to the group $Aut^{\pm}(X)$ of conformal and
anticonformal automorphisms of $X$. We want to find the groups of
automorphisms of real cyclic $p$-gonal Riemann surfaces. As for the case of groups of
conformal automorphisms of $p$-gonal Riemann surfaces we have (see {\cite{BCI}):

\begin{lemma}
\label{extensionklein} Let $(X, f, \s)$ be a real cyclic $p$-gonal Riemann
surface such that the $p$-gonality group is normal in $Aut^{\pm}(X)$. Then
$Aut^{\pm}(X)$ is an extension of $C_{p}$ by a group of conformal and
anticonformal automorphisms of the Riemann sphere.
\end{lemma}

The group of $\pm$-automorphisms of real p-gonal Riemann surfaces are known and classified according to spherical group of automorphisms of the quotient of the surface by the group of p-gonality (see \cite{CI2}). 
A finite group $\G$ of conformal and anticonformal automorphisms of
the Riemann sphere is a subgroup of:
\[ D_{q},\,C_{q}\times C_{2},\,D_{q}\rtimes C_{2},\,A_{4}\times C_{2},\,\Sigma_{4},\,\Sigma_{4}\times C_{2},\,A_{5}\times C_{2}. \]
The following theorem gives the $\pm$-automorphism groups of real cyclic $p$-gonal Riemann surfaces
\begin{thm} [\cite{BCI}]
\label{real} Let $(\X_g,f,\s)$ be a real cyclic $p$-gonal Riemann surface
with $p$ an odd prime integer, $g\geq(p-1)^{2}+1$. If the $p$-gonality group
of $\X_g$ is $\<\varphi\>$ and $\G=Aut^{\pm}
(\X_g)/\<\varphi\>$ then the possible $\pm$-automorphisms groups of
$\X_g$ are

\begin{enumerate}
\item $C_{pq}\times C_{2}$ if $\<\varphi, \s\>= C_{2p}$. 

$D_{pq}$ if $\<\varphi, \s\>= D_{p}$

\item $D_{pq}\rtimes C_{2}$, where $\rtimes$ means any possible semidirect
product (including the direct product).

\item $(C_{p}\rtimes C_{q})\rtimes C_{2},$ where $\rtimes$ means any possible
semidirect product (including the direct product).

\item $(C_{p}\rtimes D_{q})\rtimes C_{2}$, where $\rtimes$ means any possible
semidirect product (including the direct product).

\item $C_{p}\rtimes_{2}\Sigma_{4}\,D_{p}\times A_{4},\,D_{p}\times\Sigma_{4},\,D_{p}\times A_{5}$ if
$\<\varphi,\s\>=D_{p}$

$C_{p}\times\Sigma_{4},\,C_{2p}\times A_{4},\,C_{2p}\times\Sigma_{4}
,\,C_{2p}\times A_{5}$ if $\<\varphi,\s\>=C_{2p}$

\item Exceptional Case 1. $((C_{2}\times C_{2})\rtimes_{3}C_{9})\rtimes
_{2}C_{2}$ for $p=3$ and $\G=\Sigma_{4}$ where $\<
\varphi,\s\>=D_{p}$

$((C_{2}\times C_{2})\rtimes_{3}C_{9})\times C_{2}$ for $p=3$ and
$\G=A_{4}\times C_{2}$ where $\<\varphi,\s\>=C_{2p}$

\item Exceptional Case 2. $(C_{p}\times C_{2}\times C_{2})\rtimes_{3}C_{6}$ for $p\equiv 1\mod 6$, $\G=A_{4}\times C_{2}$ and
$\<\varphi,\s\>=C_{2p}$

\item Exceptional Case 3. $(((C_{2}\times C_{2})\rtimes_{3}C_{9})\rtimes
_{2}C_{2})\times C_{2}$ for $p=3$ and $\G=\Sigma_{4}\times C_{2}$.
\end{enumerate}
\end{thm}

Now, one can calculate the automorphism groups of $p$-gonal real curves. Let $\mathcal{C}$ be the $p$-gonal real curve associated to a real cyclic $p$-gonal Riemann surface $(X, f, \s)$ with symmetry $\s$. The automorphism group of the real curve $\mathcal{C}$ consists of the (conformal and anticonformal) automorphisms of $X$ commuting with the symmetry $\s$, that is the normalizer in $Aut^{\pm}(X)$ of $\< \s \>, \, N_{G}(\< \s \>)$.

Then the symmetry type and the automorphism group of $\mathcal{C}$ are studied below according to $G =Aut^{\pm}(X)$: 

\begin{enumerate}
\item  {\bf Case 1a} $G = C_{pq}\times C_{2}=\< \widehat{\varphi}, \s \>$. In this case there is one class of symmetries (one real curve)(with representative $\s$) if $q\equiv 1\mod 2 $, and two classes of symmetries - two real curves- (with representatives $\s$ and $\tau =\s\widehat{\varphi}^{pq/2}$) if $q\equiv 0\mod 2 $. In both cases the automorphism group of the corresponding real curves is the whole group $G =C_{pq}\times C_{2} = N_{G}(\< \varphi, \s \>) = N_{G}(\< \varphi, \tau \>) $

\smallskip
\item {\bf Case 1b} $D_{pq} =\< \widehat{\varphi}, \s \>$. Again, $X$ has one or two classes of symmetries according to wether $q$ is odd or even respectively (the representatives are $\s$ and $\tau =\s\widehat{\varphi}$). The automorphisms groups of the real curves associated to $(X, \s)$ and $(X, \tau)$ equal $D_{p} =  N_{G}(\< \varphi, \s \>) = N_{G}(\< \varphi, \tau \>) $

\smallskip
\item {\bf Case 2a} $G=D_{pq}\times C_{2} =\< \widehat{\varphi}, \rho \> \times \< \s  \>$, where $q\equiv 1\mod 2 $. There are two classes of symmetries (two real curves) with representatives $\s$ and $\tau =\s \rho$. 
The real curve associated to $(X, \s)$  has automorphism group $G = N_{G}(\< \varphi, \s \>)$, and the real curve associated to $(X, \tau)$ has automorphism group $D_{p} =  N_{G}(\< \varphi, \tau \>)$.

\smallskip
\item {\bf Case 2b} $G=D_{pq}\times C_{2} =\< \widehat{\varphi}, \rho \> \times \< \s  \>$, where $q\equiv 0\mod 2 $. There are four classes of symmetries (four real curves) with representatives $\s,  \, \s\widehat{\varphi}^{pq/2}, \, \tau_1=\s\rho$ and $\tau_2=\s\rho\widehat{\varphi}$. The real curves associated to $(X, \s)$  and $(X, \s\widehat{\varphi}^{pq/2})$ have automorphism group $G = N_{G}(\< \varphi, \s \>) = N_{G}(\< \varphi, \s\widehat{\varphi}^{pq/2} \>)$. The  real curves associated to $(X, \s\rho)$ and $(X, \s\rho\widehat{\varphi})$ have autoharphism group $D_{p} =  N_{G}(\< \varphi, \tau_1 \>) = N_{G}(\< \varphi, \tau_2 \>)$.

\smallskip
\item {\bf Case 2c} $G=D_{pq}\rtimes C_{2} =\< \widehat{\varphi}, \rho, \, \s \, |   \widehat{\varphi}^{p}=\rho^{2}=\s^{2}=( \widehat{\varphi}^{p}\rho)^{2}=( \widehat{\varphi}^{p}\s)^{2}=\s\rho\s\rho \widehat{\varphi}^{-j}= 1 \>$. Using elementary group theory we get that the symmetries are $\s\varphi^{t}, \, 0\le t \le p-1$ and $\s\rho\widehat{\varphi}^{j/2}$. All the symmetries $\s\varphi^{t}, \, 0\le t \le p-1$ are conjugate. The real curve associate to their conjugacy class has sutomorphism group $G = N_{G}(\< \varphi, \s \>)$. The symmetry $\s\rho\widehat{\varphi}^{j/2}$ is central in $G$, thus the automorphism group of the real curve given by $(X, \s\rho\widehat{\varphi}^{j/2})$ is the whole $G=D_{pq}\rtimes C_{2}$.

\smallskip
\item {\bf Case 3a} $G =(C_{p}\rtimes C_{q})\times C_{2}=\< \varphi, \alpha \, | \varphi^{p} =\alpha^{q} =\alpha^{-1}\varphi \alpha \varphi^{-i} =1  \> \times \< \s  \>$.  ($i^{d}\equiv 1\mod p $ where $d=mcd(q, p-1)$). By elementary group theory one obtains that, if $q \equiv 1\mod 2 $, the only symmetry is $\s$. In this case the automorphism group of the real curve is $G = N_{G}(\< \varphi, \s \>)$. Now, let $q\equiv 0\mod 2 $. If $i =-1$ and $q\equiv 2\mod 4$ the only symmetries are $\s$ and $\s \alpha^{q/2}$, they are not conjugate so they induce two non-equivalent real curves, both with automorphism group the whole group $G$. If $i \neq -1$ or $q\equiv 0\mod 4$ there are three symmetries: $\s, \, \s \alpha^{q/2}$ and $\s \alpha^{q/2}\varphi^{h}$, with $h$ the solution of $h(i^{q/2}+1 \equiv 1modp$. Moreover, the symmetries $\s \alpha^{q/2}$ and $\s \alpha^{q/2}\varphi^{h}$ are non-conjugate if and only if $i=1$ or $i = -1$ and $q\equiv 0\mod 4$. In any case the automorphism group of the corresponding real curve is the group $G$.

\smallskip
\item {\bf Case 3b} $G =(C_{p}\rtimes C_{q})\rtimes C_{2}=\< \varphi, \alpha \, | \varphi^{p} =\alpha^{q} =\alpha^{-1}\varphi \alpha \varphi^{-i} =1  \> \rtimes \< \s  \>$ with $i^{d}\equiv 1\mod p $ where $d=mcd(q, p-1)$. First of all there are three remaining possible cases for the action of $\s$ on $\varphi$ and $\alpha$: 
\begin{enumerate}
\item $\s\varphi\s=\varphi^{-1}, \, \s\alpha\s=\alpha$. Using elementary group theory one obtains that the only symmetries are $\s\varphi^{t}, \, 0\le t \le p-1$, all them in the conjugate to each other. The automorphism group of the real curve is $G = N_{G}(\< \varphi, \s \>)$.

\item $\s\varphi\s=\varphi^{-1}, \, \s\alpha\s=\alpha^{-1}$. Using elementary group theory one obtains that the symmetries are $\s\alpha^{r}, \, 0\le r \le q-1$, and $\s\alpha^{r_0}\varphi^{t}, \, 0\le t \le p-1$, where $r_0\neq 0$ is such that $i^{r_0}\equiv -1\mod p $. Now, for even $q$, there are two conjugacy classes of symmetries in $G$ with representatives $\s$ and $\s\alpha$. The automorphism groups of both real curves coincide and equal $D_{p} \rtimes C_2=  N_{G}(\< \varphi, \s \>) = N_{G}(\< \varphi, \s\alpha \>)$ (the product maybe a direct product). For odd $q$, there is a unique class of symmetries in $G$ with representative $\s$ The automorphism group of the real curve is $D_{p}=  N_{G}(\< \varphi, \s \>)$.

\item $\s\varphi\s=\varphi, \, \s\alpha\s=\alpha^{-1}$. Using elementary group theory one obtains that the symmetries are $\s\alpha^{r}, \, 0\le r \le q-1$, and $\s\alpha^{r_0}\varphi^{t}, \, 0\le t \le p-1$, where $r_0\neq 0$ is such that $i^{r_0}\equiv -1\mod p $. Now, for even $q$, there are two conjugacy classes of symmetries in $G$ with representatives $\s$ and $\s\alpha$. The automorphism groups of both real curves coincide and equal $C_{2p} \rtimes C_2=  N_{G}(\< \varphi, \s \>) = N_{G}(\< \varphi, \s\alpha \>)$ (the product maybe a direct product). For odd $q$, there is a unique class of symmetries in $G$ with representative $\s$ The automorphism group of the real curve is $C_{2p}=  N_{G}(\< \varphi, \s \>)$.
\end{enumerate}

\smallskip
\item {\bf Case 4a} $G =(C_{p}\rtimes D_{q})\times C_{2}=\< \varphi, \alpha \, \rho \, | \varphi^{p} =\alpha^{q} =\rho^{2}=\alpha^{-1}\varphi \alpha \varphi^{-i} =(\rho\alpha)^{2}=\rho\varphi\rho\varphi^{\pm 1}=1  \> \times \< \s  \>$ 
with $i^{d} \equiv 1\mod p $ where $d=mcd(q, p-1)$. We divide the case in subcases according to the action of $\rho$  on $\varphi$: 
\begin{enumerate}

\smallskip
\item $G =(C_{p}\rtimes D_{q})\times C_{2}=\< \varphi, \alpha \, \rho \, | \varphi^{p} =\alpha^{q} =\rho^{2}=\alpha^{-1}\varphi \alpha \varphi^{-i} =(\rho\alpha)^{2}=\rho\varphi\rho\varphi^{-1}=1  \> \times \< \s  \>$. $\s$ is a central symmetry in $G$, $\s\rho\alpha^{r}, \, 0\le r\le q-1$ are symmetries as well. For semidirect products such that there are $r_0$ such that $i^{r_0}\equiv -1\mod p $ the elements $\s\rho\alpha^{r_0}\varphi^{t}, \, 1\le t \le p-1$ are also symmetries.

For $q\equiv 1\mod 2 $ the above are the only symmetries. The group $G$ contains two classes of symmetries, the central symmetry $\s$ and the conjugacy class of $\s\rho$. The automorphism groups of the corresponding real curves are $G=N_{G}(\< \varphi, \s \>)$ and $C_{2p} \times C_2=  N_{G}(\< \varphi, \s\rho \>)$.

For $q\equiv 0\mod 2 $ In this case the symmetries are the ones listed above together with $\s\alpha^{q/2}$ and, if $p\equiv 1\mod q$ (equivalently $i^{q/2}\equiv -1\mod p $), also $\s\alpha^{q/2}\varphi^{t}, \, 0\le t \le p-1$. These last symmetries, if they do exist, are conjugate to $\s\alpha^{q/2}$.  The group $G$ contains four classes of symmetries, the central symmetry $\s$ and the conjugacy classes with representatives $\s\rho, \,  \s\rho\alpha$ and $\s\alpha^{q/2}$. The automorphism groups of the corresponding real curves are $G=N_{G}(\< \varphi,  \, \s \>)$ and $C_{2p} \times C_2=  N_{G}(\< \varphi, \s\rho \>)  = N_{G}(\< \varphi, \, \s\rho\alpha \>)$ and $G=N_{G}(\< \varphi,  \, \s\alpha^{q/2} \>)$.

\smallskip
\item $G =(C_{p}\rtimes D_{q})\times C_{2}=\< \varphi, \alpha \, \rho \, | \varphi^{p} =\alpha^{q} =\rho^{2}=\alpha^{-1}\varphi \alpha \varphi^{-i} =(\rho\alpha)^{2}=(\rho\varphi\rho\varphi^){2}=1  \> \times \< \s  \>$. we can assume that $i\neq -1$, since in that case we are in the conditions of Case 4a.  $\s$ is a central symmetry in $G$, $\s\rho\alpha^{r}, \, 0\le r\le q-1$ and $\s\rho\alpha^{r_0}\varphi^{t}, \, 1\le t \le p-1$ with $r_0$ such that $i^{r_0}\equiv 1\mod p $ (for instance $r_0=0$)  are also symmetries.

For $q\equiv 1\mod 2 $ the above are the only symmetries. The group $G$ contains two classes of symmetries, the central symmetry $\s$ and the conjugacy class of $\s\rho$. The automorphism groups of the corresponding real curves are $G=N_{G}(\< \varphi, \s \>)$ and $D_{p} \times C_2=  N_{G}(\< \varphi, \s\rho \>)$. 

For $q\equiv 0\mod 2 $ In this case the symmetries are the ones listed above together with $\s\alpha^{q/2}$ and, if $p\equiv 1\mod q$ (equivalently $i^{q/2}\equiv -1\mod p $), also $\s\alpha^{q/2}\varphi^{t}, \, 0\le t \le p-1$. These last symmetries, if they do exist, are conjugate to $\s\alpha^{q/2}$.  The group $G$ contains four classes of symmetries, the central symmetry $\s$ and the conjugacy classes with representatives $\s\rho, \,  \s\rho\alpha$ and $\s\alpha^{q/2}$. The automorphism groups of the corresponding real curves are $G=N_{G}(\< \varphi,  \, \s \>)$ and $D_{p} \times C_2=  N_{G}(\< \varphi, \s\rho \>)  = N_{G}(\< \varphi, \, \s\rho\alpha \>)$ and $G=N_{G}(\< \varphi,  \, \s\alpha^{q/2} \>)$.

\end{enumerate}

\smallskip
\item {\bf Case 4b} $G =(C_{p}\rtimes D_{q})\rtimes C_{2}=\< \varphi, \alpha \, \rho, \s \, | \varphi^{p} =\alpha^{q} =\rho^{2}=\s^{2}=\alpha^{-1}\varphi \alpha \varphi^{-i} =(\rho\alpha)^{2}=\rho\varphi\rho\varphi^{- 1}=(\s\varphi)^{2}=1, \,   \s\rho\s=\rho\alpha^{k}, \s\alpha\s=\alpha^{j}  \>$, 
with $ j= \pm 1, \, i^{d} \equiv 1\mod p $ where $d=mcd(q, p-1)$. We can assume that $\rho$ and $\varphi$ commute, otherwise we choose $\s\rho$ as generator of $D_q$. As $\s^{2}\rho\s^{2}=\rho$ we get that $k(1+j)\equiv 0\mod q$. So either $j=-1$ or $j=1,  k=0$; or $j=1, k=q/2$ if $q\equiv 0\mod 2 $.

We divide the case in subcases according to the action of $\s$  on $\< \rho, \, \alpha \>$: 

\begin{enumerate}

\smallskip
\item Case $j=-1$ The group $G =(C_{p}\rtimes D_{q})\rtimes C_{2}=\< \varphi, \alpha \, \rho, \s \, | \varphi^{p} =\alpha^{q} =\rho^{2}=\s^{2}=\alpha^{-1}\varphi \alpha \varphi^{-i} =(\rho\alpha)^{2}=\rho\varphi\rho\varphi^{- 1}=(\s\varphi)^{2}=1, \,   \s\rho\s=\rho\alpha^{k}, (\s\alpha)^{2}=1 \>$. The symmetries are $\s, \, \s\alpha^{r}, \, 0\le r \le q-1, \, \s\alpha^{r_0}\varphi^{t}, \, 1\le t \le p-1$, where $r_0$ satisfies that $i^{r_0}\equiv 1\mod p $. For $q\equiv 1\mod 2 $,  let $k/2$ denote the solution of $2s\equiv k\mod q$, for $q\equiv 0\mod 2 $ the parameter $k$ is even. Then the remaining symmetries are $\s\rho\alpha^{k/2}$ and, if $i^{k/2}\equiv 1\mod p $, the involutions $\s\rho\alpha^{k/2}\varphi^{t}, \, 1\le t \le p-1$. There are two conjugacy clases of symmetries if $q\equiv 1\mod 2 $; with representatives $\s$ and $\s\rho\alpha^{k/2}$. The automorphism groups of the corresponding real curves are $G=N_{G}(\< \varphi,  \, \s\rho\alpha^{k/2} \>)$ and $D_{p} \rtimes C_2=  N_{G}(\< \varphi, \s \>)$.  For $q\equiv 0 mod$  
there are three conjugacy classes of symmetries with representatives $\s$, $\s\alpha$ and $\s\rho\alpha^{k/2}$. The automorphism groups of the corresponding curves are $G=N_{G}(\< \varphi,  \, \s\rho\alpha^{k/2} \>)$ and $D_{p} \rtimes C_2=  N_{G}(\< \varphi, \s \>) = N_{G}(\< \varphi, \s\alpha \>)$.

\smallskip
\item Case $j=1, k=0$.  The group $G =(C_{p}\rtimes D_{q})\rtimes C_{2}=\< \varphi, \alpha \, \rho, \s \, | \varphi^{p} =\alpha^{q} =\rho^{2}=\s^{2}=\alpha^{-1}\varphi \alpha \varphi^{-i} =(\rho\alpha)^{2}=\rho\varphi\rho\varphi^{- 1}=(\s\varphi)^{2}=1 =  ( \s\rho)^{2}, \s\alpha\s=\alpha \>$. The symmetries are $\s, \, \s\rho\alpha^{r}, \, 0\le r \le q-1, \, \s\rho\alpha^{r_0}\varphi^{t}, \, 1\le t \le p-1$, where $r_0$ satisfies that $i^{r_0}\equiv 1\mod p $. For $q\equiv 1\mod 2 $,  let $q/2=0$ . Then the remaining symmetries are $\s\alpha^{q/2}$ and, if $i^{q/2}\equiv 1\mod p $, the involutions $\s\alpha^{k/2}\varphi^{t}, \, 1\le t \le p-1$. There are two conjugacy clases of symmetries if $q\equiv 1\mod 2 $; with representatives $\s$ and $\s\rho$. The automorphism groups of the corresponding real curves are $G=N_{G}(\< \varphi,  \, \s \>)$ and $D_{2p} \rtimes C_2=  N_{G}(\< \varphi, \s\rho \>)$.  For $q\equiv 0\mod 2 $ there are three conjugacy classes of symmetries with representatives $\s$, $\s\rho$, $\s\alpha^{q/2}$. The automorphism groups of the corresponding curves are $G=N_{G}(\< \varphi,  \, \s\alpha^{q/2} \>)$ and $D_{2p} \rtimes C_2=  N_{G}(\< \varphi, \s \>) = N_{G}(\< \varphi, \s\alpha \>)$. 

\smallskip
\item Case $q\equiv 0\mod 2 , \, j=1, k=q/2$. We have the group $G =(C_{p}\rtimes D_{q})\rtimes C_{2}=\< \varphi, \alpha \, \rho, \s \, | \varphi^{p} =\alpha^{q} =\rho^{2}=\s^{2}=\alpha^{-1}\varphi \alpha \varphi^{-i} =(\rho\alpha)^{2}=\rho\varphi\rho\varphi^{- 1}=(\s\varphi)^{2}=1, \,   \s\rho\s=\rho\alpha^{q/2}, \s\alpha\s=\alpha  \>$.  In this case the symmetries are $\s, \, \s\alpha^{q/2}$ and, if $i^{q/2}\equiv 1 \mod p$ also $\s\alpha^{k/2}\varphi^{t}, \, 1\le t \le p-1$. All the symmetries are conjugate and the automorphism group of the real curve is $D_{p}\rtimes C_{q}=N_{G}(\< \varphi, \s \>)$. 

\end{enumerate}

\smallskip
\item {\bf Case 4c} $G =(C_{p}\rtimes D_{q})\rtimes C_{2}=\< \varphi, \alpha \, \rho, \s \, | \varphi^{p} =\alpha^{q} =\rho^{2}=\s^{2}=\alpha^{-1}\varphi \alpha \varphi^{-i} =(\rho\alpha)^{2}=\rho\varphi\rho\varphi^{\pm 1}=\s\varphi\s\varphi^{-1}=1, \,   \s\rho\s=\rho\alpha^{k}, \s\alpha\s=\alpha^{j}  \>$, 
with $ j= \pm 1, \, i^{d} \equiv 1\mod p $ where $d=mcd(q, p-1)$. As $\s^{2}\rho\s^{2}=\rho$ we get that $k(1+j)\equiv 0\mod q$. So either $j=-1$ or $j=1,  k=0$; or $j=1, k=q/2$ if $q\equiv 0\mod 2 $.

We divide the case in subcases according to the action of $\s$  on $\< \rho, \, \alpha \>$: 

\begin{enumerate}

\smallskip
\item Case $j=-1$ The group $G =(C_{p}\rtimes D_{q})\rtimes C_{2}=\< \varphi, \alpha \, \rho, \s \, | \varphi^{p} =\alpha^{q} =\rho^{2}=\s^{2}=\alpha^{-1}\varphi \alpha \varphi^{-i} =(\rho\alpha)^{2}=\rho\varphi\rho\varphi^{\pm 1}=sigma\varphi\s\varphi^{-1}=1, \,   \s\rho\s=\rho\alpha^{k}, (\s\alpha)^{2}=1 \>$. The symmetries are $\s, \, \s\alpha^{r}, \, 0\le r \le q-1, \, \s\alpha^{r_0}\varphi^{t}, \, 1\le t \le p-1$, where $r_0$ satisfies that $i^{r_0}\equiv -1\mod p $. For $q\equiv 1\mod 2 $,  let $k/2$ denote the solution of $2s\equiv k\mod q$, for $q\equiv 0\mod 2 $ the parameter $k$ is even. Then the remaining symmetries are $\s\rho\alpha^{k/2}$ and, for the semidirect products satisfiyng either $i^{k/2}\equiv 1\mod p $ if $\rho\varphi\rho =\varphi$ or $i^{k/2}\equiv -1\mod p $ if $\rho\varphi\rho =\varphi^{-1}$, the involutions $\s\rho\alpha^{k/2}\varphi^{t}, \, 1\le t \le p-1$. There are two conjugacy clases of symmetries if $q\equiv 1\mod 2 $; with representatives $\s$ and $\s\rho\alpha^{k/2}$. The automorphism groups of the corresponding real curves are $G=N_{G}(\< \varphi,  \, \s\rho\alpha^{k/2} \>)$ and $C_{2p} \rtimes C_2=  N_{G}(\< \varphi, \s \>)$.  For $q\equiv 0 mod$ 
there are three conjugacy classes of symmetries with representatives $\s$, $\s\alpha$ and $\s\rho\alpha^{k/2}$. The automorphism groups of the corresponding curves are $G=N_{G}(\< \varphi,  \, \s\rho\alpha^{k/2} \>)$ and $C_{2p} \rtimes C_2=  N_{G}(\< \varphi, \s \>) = N_{G}(\< \varphi, \s\alpha \>)$.

\smallskip
\item Case $j=1, k=0$.  This case is Case 4a, where $\s$ is central in $G$.

\smallskip
\item Case $q\equiv 0\mod 2 , \, j=1, k=q/2$. We have the group $G =(C_{p}\rtimes D_{q})\rtimes C_{2}=\< \varphi, \alpha \, \rho, \s \, | \varphi^{p} =\alpha^{q} =\rho^{2}=\s^{2}=\alpha^{-1}\varphi \alpha \varphi^{-i} =(\rho\alpha)^{2}=\rho\varphi\rho\varphi^{- 1}=(\s\varphi)^{2}=1, \,   \s\rho\s=\rho\alpha^{q/2}, \s\alpha\s=\alpha  \>$.  In this case the symmetries are $\s, \, \s\alpha^{q/2}$ and, if $i^{q/2}\equiv 1 \mod p$ also $\s\alpha^{k/2}\varphi^{t}, \, 1\le t \le p-1$. All the symmetries are conjugate and the automorphism group of the real curve is $C_{2p}\rtimes C_{q}=N_{G}(\< \varphi, \s \>)$. 

\end{enumerate}

\smallskip

\item Case 5. The $\pm-$automorphism group $G=Aut^{\pm}(X)$ of the real Riemann surface $X$ is one of $C_{p}\rtimes_{2}\Sigma_{4}
,\,D_{p}\times A_{4},\,D_{p}\times\Sigma_{4},\,D_{p}\times A_{5}$ if
$\<\varphi,\s\>=D_{p}$, or $C_{p}\times\Sigma_{4},\,C_{2p}\times A_{4},\,C_{2p}\times\Sigma_{4}
,\,C_{2p}\times A_{5}$ if $\<\varphi,\s\>=C_{2p}$. 

\begin{enumerate}

\item There is one conjugacy class of symmetries in $G=Aut^{\pm}(X)= C_{p}\rtimes_2 \Sigma_4=\< \varphi \>\rtimes_2 \Sigma_4$, with $\Sigma_4=\< \tau, \alpha \, | \alpha^3=\tau^2={\alpha\tau}^4=1 \>$; with representative $\tau$. The automorphism group of these real curves  equals $C_p\times C_2\times C_2 \times C_2$.

\item There are two conjugacy classes of symmetries in $G=Aut^{\pm}(X)= D_{p}\times A_4$ or $G= C_{2p}\times A_4$, with $A_4=\< \tau, \alpha \, | \alpha^3=\tau^2={\alpha\tau}^3=1 \>$, with representatives $\s$ and $\s\tau$. The automorphism groups of the corresponding real curves are either $G$ for the real curve associated to the symmetry $\s$ and  $C_{2p}\times C_2\times C_2$ if $\<\varphi,\s\>=C_{2p}$, or $D_{p}\times C_2\times C_2$ if $\<\varphi,\s\>=D_{p}$ for real curves given by the conjugacy class of the symmetry $\s\tau$.

\item There are two conjugacy classes of symmetries in $G=Aut^{\pm}(X)= D_{p}\times A_5$ or $G= C_{2p}\times A_5$, with $A_5=\< \tau, \alpha \, | \alpha^5=\tau^2={\alpha\tau}^3=1 \>$, with representatives $\s$ and $\s\tau$.. The automorphism groups of the corresponding real curves are either $G$ for the real curve associated to the symmetry $\s$ and  $C_{2p}\times C_2$ if $\<\varphi,\s\>=C_{2p}$, or $D_{p}\times C_2$ if $\<\varphi,\s\>=D_{p}$ for real curves given by the conjugacy class of the symmetry $\s\tau$.

\item There are three conjugacy classes of symmetries in $G=Aut^{\pm}(X)= D_{p}\times \Sigma_4$ or $G= C_{2p}\times \Sigma_4$, with $\Sigma_4=\< \tau, \alpha \, | \alpha^3=\tau^2={\alpha\tau}^4=1 \>$, the conjugacy classes have representatives $\s, \, \s\tau$ and $\s(\tau\alpha)^2$. The automorphism groups of the corresponding real curves are either $G$ for the real curve associated to the central symmetry $\s$ and $C_{2p}\times C_2\times C_2\times C_2$ if $\<\varphi,\s\>=C_{2p}$, or $D_{p}\times C_2\times C_2\times C_2$ if $\<\varphi,\s\>=D_{p}$ for real curves given by the conjugacy classes of the symmetries $\s\tau$ and $\s(\tau\alpha)^2$.

\end{enumerate}

\smallskip
\item Exceptional Case 1. The $\pm$-automorphism group is $G=((C_{2}\times C_{2})\rtimes_{3}C_{9})\rtimes
_{2}C_{2} = ((\< \tau_1, \tau_2 \>)\rtimes_3 \< \widehat{\varphi} \>)\rtimes_2 \< \s \>$ for $p=3$ and $\G=\Sigma_{4}$ where $\<
\varphi,\s\>=D_{p}$ or
$G=((C_{2}\times C_{2})\rtimes_{3}C_{9})\times C_{2}= ((\< \tau_1, \tau_2 \>)\rtimes_3 \< \widehat{\varphi} \>)\times \< \s \>$ for $p=3$ and
$\G=A_{4}\times C_{2}$ where $\<\varphi,\s\>=C_{2p}$

In the first case the group $G$ has two conjugacy classes of symmetries with representatives $\s$ and $\s\widehat{\varphi}^3\tau_1$. Real curves given by the symmetry $\s$ have automorphism group $D_{3}\times C_{2}\times C_2=N_{G}(\< \widehat{\varphi}^3, \s \>)$; real curves given by the symmetry $\s\widehat{\varphi}^3\tau_1$ have automorphism group $(C_{6}\times C_2)\rtimes C_{2}=N_{G}(\< \widehat{\varphi}^3, \s\widehat{\varphi}^3\tau_1 \>)$.

In the second case the group $G$ has two conjugacy classes of symmetries with representatives $\s$ (central symmetry) and $\s\tau_1$. Real curves given by the central symmetry $\s$ have automorphism group $G=N_{G}(\< \widehat{\varphi}^3, \s \>)$; real curves given by the symmetry $\s\tau_1$ have automorphism group $(C_{6}\times C_2)\times C_{2}=N_{G}(\< \widehat{\varphi}^3, \s\tau_1 \>)$.

\item Exceptional Case 2. $G=(C_{p}\times C_{2}\times C_{2})\rtimes_{3}C_{6}=\< \varphi, \tau_1, \tau_2 \> \rtimes_{3}\< \alpha \>, \, \s=\alpha^{3}$
for $p\equiv1\mod 6$, $\G=A_{4}\times C_{2}$ and
$\<\varphi,\s\>=C_{2p}$. Again $G$ two conjugacy classes of symmetries, the central symmetry $\s$ and the conjugacy class with representative $\s\tau_1$. Real curves given by the central symmetry $\s$ have automorphism group $G=N_{G}(\< \varphi, \s \>)$; real curves given by the symmetry $\s\tau_1$ have automorphism group $(C_{p}\times C_2)\times C_{2}\times C_2=N_{G}(\< \varphi, \s\tau_1 \>)$.

\item Exceptional Case 3. $G=(((C_{2}\times C_{2})\rtimes_{3}C_{9})\rtimes
_{2}C_{2})\times C_{2} = ((\< \tau_1, \tau_2 \> \rtimes_{3}\< \widehat{\varphi} \>)\rtimes
_{2}\< \alpha \>)\times \< \s \>$ for $p=3$ and $\G=\Sigma_{4}\times C_{2}$. 
There are three conjugacy classes of symmetries in $G$,with representatives $\s, \, \s\alpha$ and $\s\alpha\widehat{\varphi}^3\tau_1$.
 The automorphism groups of the corresponding real curves are either $G$ for the real curve associated to the central symmetry $\s$, $(C_{6}\times C_2\times C_2)\rtimes C_2$  for real curves given by the conjugacy class of the symmetry $\s\alpha\widehat{\varphi}^3\tau_1$, and $D_{3}\times C_2\times C_2\times C_2$  for real curves given by the conjugacy class of the symmetry $\s\alpha$.
\end{enumerate}

Notice that as in the case of real hyperelliptic surfaces to provide defining equations one has to look not only to the automorphism groups but also to the ramification of the $p$-gonal morphism. When calculating equations of $p$-gonal Riemann surfaces the ramification set on the sphere was invariant by the complex conjugation, in other words, the equations already found are equations for the real sublocus of the $p$-gonal locus.


\end{document}